\documentclass[11pt]{amsart}
\usepackage{a4wide,hyperref}

\newcommand{\w}{\omega}

\newcommand{\IN}{\mathbb N}
\newcommand{\IR}{\mathbb R}
\newcommand{\e}{\varepsilon}
\newcommand{\Ra}{\Rightarrow}

\newtheorem{theorem}{Theorem}
\newtheorem{corollary}{Corollary}
\newtheorem{lemma}{Lemma}

\newtheorem{example}{Example}
\newtheorem{proposition}{Proposition}

\theoremstyle{definition}
\newtheorem{remark}{Remark}

\title[The completion of the hyperspace of finite subsets]{The completion of the hyperspace of finite subsets, endowed with the $\ell^1$-metric}
\author[I.~Banakh, T.~Banakh, J.~Garbuli\'nska-W\c egrzyn]{Iryna Banakh$^1$, Taras Banakh$^{2,3}$ and  Joanna Garbuli\'nska-W\c egrzyn$^{3}$}
\address{$^{1}$Pidstryhach Institute for Applied Problems of Mechanics and Mathematics, National Academy of Sciences of Ukraine, Lviv, Naukova 3b, Ukraine}
\email{ibanakh@yahoo.com}  
\address{$^{2}$Ivan Franko National University of Lviv, Ukraine}
\email{t.o.banakh@gmail.com}
\address{$^{3}$Institute of Mathematics, Jan Kochanowski University, Kielce, Poland}
\email{jgarbulinska@ujk.edu.pl}
\begin{document}

\begin{abstract} For a metric space $X$, let $\mathsf FX$ be the space of all nonempty finite subsets of $X$ endowed with the largest metric $d^1_{\mathsf FX}$ such that for every $n\in\IN$ the map $X^n\to\mathsf FX$, $(x_1,\dots,x_n)\mapsto \{x_1,\dots,x_n\}$, is  non-expanding with respect to the $\ell^1$-metric on $X^n$.
 We study the completion of the metric space $\mathsf F^1\!X=(\mathsf FX,d^1_{\mathsf FX})$ and prove that it coincides with the space $\mathsf Z^1\!X$ of nonempty compact subsets of $X$ that have zero length (defined with the help of graphs). We prove that each subset of zero length in a metric space has 1-dimensional Hausdorff measure zero. A subset $A$ of the real line has zero length if and only if its closure is compact and has Lebesgue measure zero. On the other hand, for every $n\ge 2$ the Euclidean space $\IR^n$ contains a compact subset of 1-dimensional Hausdorff measure zero that fails to have zero length.
\end{abstract}
\maketitle

\section{Introduction}

Given a metric space $X$ with metric $d_X$, denote by $\mathsf KX$ the space of all nonempty compact subsets of $X$, endowed with the Hausdorff metric $d_{\mathsf KX}$ defined by the formula
$$d_{\mathsf KX}(A,B)=\max\{\max_{a\in A}\min_{b\in B}d_X(a,b),\max_{b\in B}\min_{a\in A}d_X(b,a)\}.$$ The metric space $\mathsf KX$, called the {\em hyperspace} of $X$, plays an important role in  General Topology \cite[\S3.2]{Beer}, \cite[4.5.23]{Eng} and Theory of Fractals \cite[\S2.5]{Edgar}, \cite[\S9.1]{Fal}. It is well-known \cite[4.5.23]{Eng} that for any complete (and compact) metric space $X$ its hyperspace $\mathsf KX$ is complete (and compact). The hyperspace $\mathsf KX$ contains an important dense subspace $\mathsf FX$ consisting of nonempty finite subsets of $X$. The density of $\mathsf FX$ in $\mathsf KX$ implies that for a complete metric space $X$, the hyperspace $\mathsf KX$ is a completion of the hyperspace $\mathsf FX$.

In \cite[\S30]{BBKZ} it was shown that the Hausdorff metric $d_{\mathsf FX}$ on $\mathsf FX$ coincides with the largest metric on $\mathsf FX$ such that for every $n\in\IN$ the map $X^n\to \mathsf FX$, $x\mapsto x[n]:=\{x(i):i\in n\}$, is non-expanding, where $X^n$ is endowed with the $\ell^\infty$-metric
$$d^\infty_{X^n}(x,y)=\max_{i\in n}d_X(x(i),y(i)).$$
Here we identify the natural number $n$ with the set $\{0,\dots,n-1\}$ and think of the elements of $X^n$ as functions $x:n\to X$.

Let us recall that a function $f:Y\to Z$ between metric spaces $(Y,d_Y)$ and $(Z,d_Z)$ is {\em non-expanding} if $d_Z(f(y),f(y'))\le d_Y(y,y')$ for any $y,y'\in Y$. 

It is well-known that the $\ell^\infty$-metric $d^\infty_{X^n}$ on $X^n$ is the limit at $p\to\infty$ of the $\ell^p$-metrics $d^p_{X^n}$ on $X^n$, defined by the formula:
$$d^p_{X^n}(x,y)=\Big(\sum_{i=1}^nd_X(x(i),y(i))^p\Big)^{\frac1p}\mbox{ \ for \ $x,y\in X^n$.}$$

Given any metric space $(X,d)$ and any number $p\in[1,\infty]$, let $d^p_{\mathsf FX}$ be the largest metric $d^p_{\mathsf FX}$ on the set $\mathsf FX$ such that for every $n\in\IN$ the map $X^n\to \mathsf FX$, $x\mapsto x[n]$, is non-expanding with respect to the $\ell^p$-metric $d^p_{X^n}$ on $X^n$.  
The metric $d^p_{\mathsf FX}$ was introduced in \cite{BBKZ}, where it was shown  that $d^p_{\mathsf FX}$ is a well-defined metric on $\mathsf FX$ such that 
$$d_{\mathsf FX}=d^\infty_{\mathsf FX}\le d^p_{\mathsf FX}\le d^1_{\mathsf FX},$$
where $d_{\mathsf FX}$ stands for the Hausdorff metric on $\mathsf FX$. 

By $\mathsf F^p\!X$ we will denote the metric space $(\mathsf FX,d^p_{\mathsf FX})$. So, $\mathsf F^\infty\! X$ coincides with the hyperspace $\mathsf FX$ endowed with the Hausdorff metric.

As we already know, for any complete metric space $X$, the completion $\hat {\mathsf F}^\infty\!X$ of the metric space $\mathsf F^\infty\!X$ can be identified with the hyperspace $\mathsf KX$ endowed with the Hausdorff metric. In this paper we study the completion $\hat {\mathsf F}^1\!X$ of the metric space $\mathsf F^1\!X=(\mathsf FX,d^1_{\mathsf FX})$ and show that it can be identified with the space $\mathsf Z^1\!X$ of nonempty compact subsets of zero length in $X$.  Sets of zero length are defined with the help of graphs.

By a {\em graph} we understand a pair $\Gamma=(V,E)$ consisting of a set $V$ of vertices and a set $E$ of edges. Each edge $e\in E$ is a nonempty subset of $V$ of cardinality $|e|\le 2$. A graph $(V,E)$ is {\em finite} if its set of vertices $V$ is finite. In this case the set of edges $E$ is finite, too.

For a graph $\Gamma=(V,E)$, a subset $C\subseteq V$ is {\em connected} if for any vertices $x,y\in C$ there exists a sequence of vertices $c_0,\dots,c_n\in C$ such that $c_0=x$, $c_n=y$ and $\{c_{i-1},c_i\}\in E$ for every $i\in\{1,\dots,n\}$. The maximal connected subsets of $V$ are called the {\em connected components} of the graph $\Gamma$. It is easy to see that two connected components of $\Gamma$ either coincide or are disjoint. For a vertex $x\in V$ by $\Gamma(x)$ we shall denote the unique connected component of the graph $\Gamma$ that contains the point $x$.

By a {\em graph in a metric space} $(X,d_X)$ we understand any graph $\Gamma=(V,E)$ with $V\subseteq X$. In this case we can define the {\em total length} $\ell(\Gamma)$ of $\Gamma$ by the formula
$$\ell(\Gamma)=\sum_{\{x,y\}\in E}d_X(x,y).$$ If $E$ is infinite, then by $\sum\limits_{\{x,y\}\in E}d_X(x,y)$ we understand the (finite or infinite) number $$\sup\limits_{E'\in \mathsf FE}\sum\limits_{\{x,y\}\in E'}d_X(x,y).$$ For a subset $C\subseteq X$ by $\overline C$ we denote the closure of $C$ in the metric space $(X,d_X)$.

Given a subset $A$ of a metric space $X$, denote by $\mathbf \Gamma_{\!X\!}(A)$ the family of graphs $\Gamma=(V,E)$ with finitely many connected components such that $V\subseteq X$ and $A\subseteq\overline V$. Observe that the family $\mathbf \Gamma_{\!X\!}(A)$ contains the complete graph on the set $A$ and hence $\mathbf \Gamma_{\!X\!}(A)$ is not empty. 

The set $A$ is defined to have {\em zero length in} $X$ if for any $\e>0$ there exists a graph $\Gamma\in\mathbf \Gamma_{\!X\!}(A)$ of total length $\ell(A)<\e$. 

In Proposition~\ref{p:zero} we shall prove that each set $A$ of zero length in a metric space $X$ is totally bounded and has 1-dimensional Hausdorff measure equal to zero. 

For a metric space $X$, denote by $\mathsf ZX$ the family of nonempty compact subsets of zero length in $X$. It is clear that each finite subset of $X$ has zero length, so $\mathsf FX\subseteq \mathsf ZX\subseteq \mathsf KX$.

Now we define the metric $d^1_{\mathsf ZX}$ on the set $\mathsf ZX$. Given two compact sets $A,B\in \mathsf ZX$, let ${\mathbf\Gamma}_{\!X\!}(A,B)$ be the family of graphs $\Gamma=(V,E)$ in $X$ such that 
\begin{itemize}
\item[(i)] $A\cup B\subseteq\overline{V}$;
\item[(ii)] $\Gamma$ has finitely many connected components;
\item[(iii)] for every connected component $C$ of $\Gamma$ we have $A\cap\overline{C}\ne\emptyset\ne B\cap\overline C$.
\end{itemize}
The conditions (i),(ii) imply that $A\cup B\subseteq\overline{V}=\bigcup_{x\in V}\overline{\Gamma(x)}$. 

Observe that the family $\mathbf \Gamma_{\!X\!}(A,B)$ contains the complete graph on the set $A\cup B$ and hence is not empty. 

For two compact subsets $A,B\in \mathsf ZX$, let 
$$d^1_{\mathsf ZX}(A,B):=\inf_{\Gamma\in\mathbf \Gamma_{\!X\!}(A,B)}\ell(\Gamma).$$

By a {\em completion} of a metric space $X$ we understand any complete metric space containing $X$ as a dense subspace. 
The following theorem is the main result of this paper.

\begin{theorem}\label{t:main} Let $X$ be a metric space and $d_X$ be its metric.
\begin{enumerate}
\item The function $d^1_{\mathsf ZX}$ is a well-defined metric on $\mathsf ZX$.
\item $d_{\mathsf KX}(A,B)\le d^1_{\mathsf ZX}(A,B)$ for any $A,B\in \mathsf ZX$.
\item $d^1_{\mathsf ZX}(A,B)=d^1_{\mathsf FX}(A,B)$ for any finite sets $A,B\in \mathsf FX$.
\item $\mathsf FX$ is a dense subset in the metric space $\mathsf Z^1\!X:=(\mathsf ZX,d^1_{\mathsf ZX})$.
\item If the metric space $X$ is complete, then so is the metric space $\mathsf Z^1\!X=(\mathsf ZX,d^1_{\mathsf ZX})$.
\item If $Y$ is a dense subspace in $X$, then $d^1_{\mathsf  ZY}(A,B)=d^1_{\mathsf ZX}(A,B)$ for any $A,B\in \mathsf ZY$.
\item If $\bar X$ is a completion of the metric space $X$, then $\mathsf Z^1\!\bar X$ is a completion of the metric space $\mathsf F^1\!X$.
\end{enumerate}
\end{theorem}

The proof of Theorem~\ref{t:main} is divided into seven lemmas.

\begin{lemma}\label{l:H} $d_{\mathsf KX}(A,B)\le d^1_{\mathsf ZX}(A,B)$ for any $A,B\in \mathsf ZX$.
\end{lemma}

\begin{proof}  To derive a contradiction, assume that $d_{\mathsf KX}(A,B)>d^1_{\mathsf ZX}(A,B)$ for some compact sets $A,B\in \mathsf ZX$. By the definition of $d^1_{\mathsf ZX}$, there exists a graph $\Gamma\in\mathbf \Gamma_{\!X\!}(A,B)$ such that $\ell(\Gamma)<d_{\mathsf KX}(A,B)$. Choose a positive real number $\e$ such that $\ell(\Gamma)+2\e<d_{\mathsf KX}(A,B)$. Since $\Gamma$ has finitely many connected components and $A\cup B\subseteq\overline V$, for any point $a\in A$ there exists a connected component $C$ of the graph $\Gamma$ such that $a\in\overline C$ . By the definition of the family $\mathbf \Gamma_{\!X\!}(A,B)$, the intersection $\overline C\cap B$ contains some point $b'\in B$. Since $a,b'\in\overline C$, there are points $c,c'\in C$ such that $d_X(a,c)<\e$ and $d_X(b',c')<\e$. Since the set $C$ is connected in the graph $\Gamma=(V,E)$, there exists a sequence $c_0,\dots,c_n\in C$ of pairwise distinct points such that $c_0=c$, $c_n=c'$, and $\{c_{i-1},c_i\}\in E$ for all $i\in\{1,\dots,n\}$. Since the points $c_0,\dots,c_n$ are pairwise distinct, the edges $\{c_{0},c_1\},\{c_1,c_2\},\dots,\{c_{n-1},c_n\}$ of the graph $\Gamma$ are pairwise distinct and then 
$$d_X(a,b')\le d_X(a,c_0)+\sum_{i=1}^nd_X(c_{i-1},c_i)+d_X(c_n,c')<\e+\ell(\Gamma)+\e.$$
Then $\min_{b\in B}d_X(a,b)\le d_X(a,b')< 2\e+\ell(\Gamma)$ and $\max_{a\in A}\min_{b\in B}<2\e+\ell(\Gamma)$. By analogy we can prove that $\max_{b\in B}\min_{a\in A}d_X(b,a)<2\e+\ell(\Gamma)$.
Then $$d_{\mathsf KX}(A,B)=\max\{\max_{a\in A}\min_{b\in B}d(a,b),\max_{b\in B}\min_{a\in A}d(b,a)\}<2\e+\ell(\Gamma)<d_{\mathsf KX}(A,B),$$
which is a desired contradiction completing the proof of the lemma.
\end{proof}

\begin{lemma}\label{l:metric} $d^1_{\mathsf ZX}$ is a well-defined metric on $\mathsf ZX$.
\end{lemma}

\begin{proof} Given any sets $A,B,C\in\mathsf ZX$, we need to verify the three axioms of metric:
\begin{enumerate}
\item $0\le d^1_{\mathsf ZX}(A,B)<\infty$ and $d^1_{\mathsf ZX}(A,B)=0$ iff $A=B$,
\item $d^1_{\mathsf ZX}(A,B)=d_{\mathsf ZX}^1(B,A)$,
\item $d^1_{\mathsf ZX}(A,B)\le d^1_{\mathsf ZX}(A,C)+d^1_{\mathsf ZX}(C,B)$.
\end{enumerate}
\smallskip

1. First we show that $d^1_{\mathsf ZX}(A,A)=0$ for any $A\in \mathsf ZX$. Since the set $A$ has zero length, for any $\e>0$ there exists a graph $\Gamma=(V,E)$ in $X$ with finitely many connected components such that $A\subseteq\overline V$ and $\ell(\Gamma)<\e$. Replacing $\Gamma$ by a suitable subgraph, we can assume that the closure of each connected component of $\Gamma$ intersects the set $A$. Then $A\in\mathbf \Gamma_{\!X\!}(A,A)$ and hence
$$d_{\mathsf ZX}^1(A,A)\le\ell(\Gamma)<\e.$$
Since $\e>0$ was arbitrary, $d_{\mathsf ZX}^1(A,A)=0$.

If sets $A,B\in \mathsf ZX$ are distinct, then by Lemma~\ref{l:H}, $d_{\mathsf ZX}^1(A,B)\ge d_{\mathsf KX}(A,B)>0$ (as the Hausdorff metric $d_{\mathsf KX}$ is a metric).
 
The proof of the first axiom of metric will be complete as soon as we check that $d^1_{\mathsf ZX}(A,B)<\infty$ for any $A,B\in \mathsf ZX$. Since the sets $A,B$ have zero length, there exist graphs $\Gamma_A=(V_A,E_A)$ and $\Gamma_B=(V_B,E_B)$ with finitely many connected components such that $A\subseteq\overline V_{\!A}$, $B\subseteq\overline V_{\!B}$ and $\ell(\Gamma_A)+\ell(\Gamma_B)<1$. Let $D$ be a finite subset of $V_A\cup V_B$ intersecting every connected component of the graphs $\Gamma_A$ and $\Gamma_B$. Consider the graph $\Gamma=(V,E)$ where $V=V_A\cup V_B$ and $E=E_A\cup E_B\cup E_D$ where $E_D:=\{e\subseteq D:|e|=2\}$. It is easy to see that the graph $\Gamma$ is connected and belongs to the family $\mathbf \Gamma_{\!X\!}(A,B)$. Then $$
d^1_{\mathsf ZX}(A,B)\le \ell(\Gamma)\le \ell(\Gamma_A)+\ell(\Gamma_B)+\sum_{\{x,y\}\in E_D}d_X(x,y)<\infty.$$

2. The definition of the distance $d^1_{\mathsf ZX}$ implies that $d^1_{\mathsf ZX}(A,B)=d^1_{\mathsf ZX}(B,A)$ for any $A,B\in \mathsf ZX$.
\smallskip

3. Finally we check the triangle inequality for $d^1_{\mathsf ZX}$. Given any $A,B,C\in \mathsf ZX$ and $\e>0$, it suffices to show that 
$$d^1_{\mathsf ZX}(A,C)\le d^1_{\mathsf ZX}(A,B)+d^1_{\mathsf ZX}(B,C)+4\e.$$ By the definition of the distances $d^1_{\mathsf ZX}(A,B)$ and $d^1_{\mathsf ZX}(B,C)$, there exist  graphs $\Gamma\in\mathbf \Gamma_{\!X\!}(A,B)$ and $\Gamma'\in\mathbf \Gamma_{\!X\!}(B,C)$ such that $\ell(\Gamma)<d^1_{\mathsf ZX}(A,B)+\e$ and $\ell(\Gamma')<d^1_{\mathsf ZX}(B,C)+\e$. By the definition of the families $\mathbf \Gamma_{\!X\!}(A,B)$ and $\mathbf\Gamma_{\!X\!}(B,C)$, the graphs $\Gamma=(V,E)$ and $\Gamma'=(V',E')$ have finitely many connected components and their closures meet the sets $A,B$ and $B,C$, respectively. 


Fix a finite set $D\subseteq V$ intersecting all connected components of the graph $\Gamma$ and a finite set $D'\subseteq V'$ intersecting all connected components of the graph $\Gamma'$. Fix a function $f:D\to B$ assigning to each point $x\in D$ a point $f(x)\in B\cap\overline{\Gamma(x)}$. Since $B\subseteq\overline V=\bigcup_{x\in V}\overline{\Gamma(x)}$,  for every $b\in B$ there exists a point $g(b)\in V$ such that  $b\in \overline{\Gamma(g(b))}$. Since $b\in \overline{\Gamma(g(b))}$ we can replace $g(b)$ by a suitable point in the connected component $\Gamma(g(b))$ and additionally assume that $d(b,g(b))<\e/{|D|}$. Next, do the same for the graph $\Gamma'$: choose a function $f':D'\to B$ such that  $f(x)\in B\cap\overline{\Gamma'(x)}$ for every $x\in D'$, and a function $g':B\to V'$ such that $b\in \overline{\Gamma'(g'(b))}$ and $d(b,g'(b))<\e/|D'|$ for every $b\in B$.
Consider the graph $\Gamma''=(V'',E'')$ where $V''=V\cup V'$ and $$E''=E\cup E'\cup\big\{\{f(x),g'(f(x))\}:x\in D\big\}\cup\big\{\{f'(x),g(f'(x))\}:x\in D'\big\}.$$ It can be shown that $\Gamma''\in\mathbf \Gamma_{\!X\!}(A,C)$ and hence
\begin{multline*}
d^1_{\mathsf ZX}(A,C)\le\ell(\Gamma'')\le\ell(\Gamma)+\ell(\Gamma')+\sum_{x\in D}d\big(f(x),g'(f(x))\big)+\sum_{x\in D'}d\big(f'(x),g(f'(x))\big)<\\
\big(d^1_{\mathsf ZX}(A,B)+\e\big)+\big(d^1_{\mathsf ZX}(B,C)+\e\big)+|D|\cdot\frac{\e}{|D|}+|D'|\cdot\frac{\e}{|D'|}=d^1_{\mathsf ZX}(A,B)+d^1_{\mathsf ZX}(B,C)+4\e.
\end{multline*}
\end{proof}

Given any finite sets, $A,B\in \mathsf FX$, let $\mathbf \Gamma^{\mathsf f}_{\!X\!}(A,B)$ be the subfamily of finite graphs in $\mathbf \Gamma_{\!X\!}(A,B)$. 

\begin{lemma}\label{l:BBKZ} $d^1_{\mathsf ZX}(A,B)=d^1_{\mathsf FX}(A,B)=\inf\limits_{\Gamma\in\mathbf \Gamma^{\mathsf f}_{\!X\!}(A,B)}\ell(\Gamma)$ for all $A,B\in \mathsf FX$.
\end{lemma}

\begin{proof} Fix any finite sets $A,B\in \mathsf FX$ and put $I=\inf\limits_{\Gamma\in\mathbf \Gamma_{\!X\!}(A,B)}\ell(\Gamma)$ and $I_{\mathsf f}=\inf\limits_{\Gamma\in\mathbf \Gamma^{\mathsf f}_{\!X\!}(A,B)}\ell(\Gamma)$. The equality $d^1_{\mathsf FX}(A,B)=I_{\mathsf f}$ was proved in Theorem~30.4 in \cite{BBKZ}. So, it suffices to show that $I=I_{\mathsf f}$. The inequality $I\le I_{\mathsf f}$ is trivial and follows from the inclusion $\mathbf \Gamma^{\mathsf f}_{\!X\!}(A,B)\subseteq \mathbf \Gamma_{\!X\!}(A,B)$. The inequality $I_{\mathsf f}\le I$ will follow as soon as we show that $I_{\mathsf f}\le I+5\e$ for any $\e>0$. Given any $\e>0$, find a graph $\Gamma\in\mathbf\Gamma_{\!X\!}(A,B)$ such that $\ell(\Gamma)<I+\e$.

By the definition of the family $\mathbf\Gamma_{\!X\!}(A,B)$, for every $a\in A$ we can find a point $v(a)\in V$ such that $a\in \overline{\Gamma(v(a))}$ and $B\cap\overline{\Gamma(v(a))}$ contains some point $\beta(a)$. Since $\beta(a)\in\overline{\Gamma(v(a))}$, there exists a point $u(a)\in \Gamma(v(a))$ such that $d_X(u(a),\beta(a))<\e/|A|$.  Since $a\in \overline{\Gamma(f(x))}$, we can replace $v(a)$ by a suitable point in the connected component $\Gamma(v(a))$ and additionally assume that $d_X(a,v(a))<\e/|A|$. Since the points $v(a),u(a)$ belong to the same connected component of the graph $\Gamma$, there exist a number $n_a\in\IN$ and a sequence $v_0(a),\dots,v_{n_a}(a)\in V$ such that $v_0(a)=v(a)$, $v_{n_a}(a)=u(a)$ and $\{v_{i-1}(a),v_i(a)\}\in E$ for every $i\in\{1,\dots,n_a\}$.

Now do the same with the set $B$: for every point $b\in B$ choose points $\alpha(b)\in A$ and $v'(b),u'(b)\in V$ such that $b\in \overline{\Gamma(v'(b))}$, $\alpha(b)\in A\cap\overline{\Gamma(v'(b))}$, $d_X(b,v'(b))<\e/|B|$, $u'(b)\in\Gamma(v'(b))$,  and $d_X(\alpha(b),u'(b))<\e/|B|$. Since the points $v'(b),u'(b)$ belong to the same connected component of the graph $\Gamma$, there exist $m_a\in\IN$ and a sequence $v_0'(b),\dots,v'_{m_b}(b)\in V$ such that $v'_0(b)=v'(b)$, $v'_{m_b}(b)=u'(b)$ and $\{v'_{i-1}(b),v'_i(b)\}\in E$ for every $i\in\{1,\dots,m_a\}$.

Now consider the finite graph $\Gamma'=(V',E')$ with the set of vertices
$$V'=A\cup B\cup\bigcup_{a\in A}\{v_i(a):1\le i\le n_a\}\cup \bigcup_{b\in B}\{v'_i(b):1\le i\le m_a\}$$ and the set of edges
\begin{multline*}
E'=\Big(\bigcup_{a\in A}\big\{\{a,v(a)\},\{u(a),\beta(a)\},\{v_{i-1}(a),v_i(a)\}:1\le i\le n_a\big\}\Big)\cup\\
\Big(\bigcup_{b\in B}\big\{\{b,v'(b)\},\{u'(b),\alpha(b)\},\{v'_{i-1}(b),v'_i(b)\}:1\le i\le m_a\big\}\Big).
\end{multline*}
It is easy to see that $\Gamma'\in\mathbf\Gamma^{\mathsf f}_{\!X\!}(A,B)$ and hence
\begin{multline*}
I_{\mathsf f}\le \ell(\Gamma')\le\\ \ell(\Gamma)+\sum_{a\in A}\big(d_X(a,v(a))+d_X(u(a),\beta(a))\big)+\sum_{b\in B}\big(d_X(b,v'(b))+d_X(\alpha(b),u'(b))\big)<\\
I+\e+2\e+2\e=I+5\e.
\end{multline*}
\end{proof}

\begin{lemma}\label{l:dense}  For any dense subset $Y\subseteq X$, the set $\mathsf FY$ is dense in the metric space $\mathsf Z^1\!X=(\mathsf ZX,d^1_{\mathsf ZX})$.
\end{lemma}

\begin{proof} Given any $A\in \mathsf ZX$ and $\e>0$, it suffices to find a set $B\in\mathsf FY$ such that $d^1_{\mathsf ZX}(A,B)<2\e$. Since $\ell(A)=0$, there exists a graph $\Gamma=(V,E)$ in $X$ such that $\Gamma$ has finitely many connected components, $A\subseteq \overline{V}$ and $\ell(A)<\e$. Choose a finite set $B'\subseteq V$ that meets each connected component of the graph $\Gamma$ and consider the subset $B''=\{b\in B':\overline{\Gamma(b)}\cap A\ne\emptyset\}$. It is easy to see that $\Gamma\in\mathbf \Gamma_{\!X\!}(A,B'')$ and hence $d^1_{\mathsf ZX}(A,B'')\le\ell(\Gamma)<\e$. 

Using the density of the set $Y$ in $X$, choose a finite set $B\subseteq Y$ and a surjective function $f:B''\to B$ such that $d_X(x,f(x))<\e/|B''|$ for all $x\in B''$. Consider the graph $\Gamma'=(V',E')$ with the set of vertices $V'=B''\cup f(B'')$ and the set of edges $E'=\{\{x,f(x)\}:x\in B''\}$. Observe that $\Gamma'\in\mathbf \Gamma_{\!X\!}(B'',B)$ and hence $d^1_{\mathsf ZX}(B,B'')\le\ell(\Gamma')<\sum_{x\in B''}d_X(x,f(x))<\e$. Then $$d^1_{\mathsf ZX}(A,B)\le d^1_{\mathsf ZX}(A,B'')+d^1_{\mathsf ZX}(B'',B)<\e+\e=2\e.$$
\end{proof}

\begin{lemma}\label{l:complete} If the metric space $X$ is complete, then so is the metric space $\mathsf Z^1\!X$.
\end{lemma}

\begin{proof} We need to prove that each Cauchy sequence in the space $\mathsf Z^1\!X$ is convergent. Since the space $\mathsf F^1\!X$ is dense in $\mathsf Z^1\!X$ (see Lemmas~\ref{l:BBKZ}, \ref{l:dense}), it suffices to prove that each Cauchy sequence in $\mathsf F^1\!X$ converges to some set $A\in \mathsf ZX$. So, fix a Cauchy sequence $\{A_n\}_{n\in\w}\subseteq\mathsf F^1\!X$. Since $d_{\mathsf FX}=d^\infty_{\mathsf FX}\le d^1_{\mathsf FX}$, the sequence $(A_n)_{n\in\w}$ remains Cauchy in the Hausdorff metric $d_{\mathsf FX}$. By the completeness of the hyperspace $\mathsf KX$, the sequence $(A_n)_{\in\w}$ converges (in the Hausdorff metric $d_{\mathsf KX}$) to some nonempty compact set $A\in \mathsf KX$. It remains to show that $A\in \mathsf ZX$ and the sequence $(A_n)_{n\in\w}$ converges to $A$ in the metric space $\mathsf Z^1\!X$.

Given any $\e>0$, use the Cauchy property of the sequence $(A_n)_{n\in\w}$ and find an increasing number sequence $(n_k)_{k\in\w}$ such that $$d^1_{\mathsf FX}(A_{n_k},A_i)<\frac\e{2^{k+1}}$$for any $k\in\w$ and $i\ge n_k$. By Lemma~\ref{l:BBKZ}, for every $k\in\w$ there exists a graph $\Gamma_k\in\mathbf \Gamma^{\mathsf f}_{\!X\!}(A_{n_k},A_{n_{k+1}})$ such that $\ell(\Gamma_k)<\frac\e{2^{k+1}}$.  Now consider the graph $\Gamma=(V,E)$ with $V=\bigcup_{k\in\w}V_k$ and $E=\bigcup_{k\in\w}E_k$ and observe that each connected component of the graph $\Gamma$ meets the finite set $A_{n_0}$, which implies that $\Gamma$ has finitely many connected components. Taking into account that $A$ is the limit of the sequence $(A_{n_k})_{k\in\w}$ in the Hausdorff metric, we conclude that $A\subseteq\overline{\bigcup_{k\in\w}A_{n_k}}\subseteq\overline V$ and the closure of each connected component of $\Gamma$ meets the set $A$. Then $\Gamma\in\mathbf \Gamma_{\!X\!}(A)$ and $$\ell(A)\le\ell(\Gamma)\le\sum_{k\in\w}\ell(\Gamma_k)<\sum_{k\in\w}\frac\e{2^{k+1}}=\e.$$
This shows that $\ell(A)=0$ and $A\in \mathsf ZX$. 

It remains to show that the sequence $(A_n)_{n\in\w}$ converges to $A$ in the metric space $\mathsf Z^1\!X$. Since this sequence is Cauchy, it suffices to show that the subsequence $(A_{n_k})_{k\in\w}$ converges to $A$. For every $k\in\w$,  consider the graph $\widetilde\Gamma_k=(\widetilde V_k,\widetilde E_k)$ with the set of vertices $\widetilde V_k=\bigcup_{i=k}^\infty V_k$ and the set of edges $\widetilde E_k=\bigcup_{i=k}^\infty E_k$. It can be shown that $\widetilde\Gamma_k\in\mathbf \Gamma_{\!X\!}(A,A_{n_k})$ and hence
$$d^1_{\mathsf ZX}(A,A_{n_k})\le\ell(\widetilde\Gamma_k)\le\sum_{i=k}^\infty\ell(\Gamma_i)<\sum_{i=k}^\infty\frac\e{2^{i+1}}=\frac\e{2^k}\underset{k\to\infty}{\;\longrightarrow\;} 0,$$which means that the sequence $(A_{n_k})_{k\in\w}$ converges to $A$ in the metric space $\mathsf Z^1\!X$.
\end{proof}

\begin{lemma}\label{l:equal} If $Y$ is a dense subspace of $X$, then $d^1_{\mathsf ZX}(A,B)=d^1_{\mathsf ZY}(A,B)$ for every $A,B\in \mathsf ZY$.
\end{lemma}

\begin{proof} The inequality $d^1_{\mathsf ZX}(A,B)\le d^1_{\mathsf ZY}(A,B)$ is trivial and follows from the inclusion\break $\mathbf \Gamma_{\!Y\!}(A,B)\subseteq \mathbf \Gamma_{\!X\!}(A,B)$.

Assuming that $d^1_{\mathsf ZX}(A,B)<d^1_{\mathsf ZY}(A,B)$, find $\e>0$ such that $d^1_{\mathsf ZX}(A,B)+7\e<d^1_{\mathsf ZY}(A,B)$. Using Lemma~\ref{l:dense}, choose finite sets $A',B'\in \mathsf FY$ such that $d^1_{\mathsf ZY}(A,A')<\e$ and $d^1_{\mathsf ZY}(B,B')<\e$. Then also $d^1_{\mathsf ZX}(A,A')\le d^1_{\mathsf ZY}(A,A')<\e$ and $d^1_{\mathsf ZX}(B,B')\le d^1_{\mathsf ZY}(B,B')<\e$. Applying the triangle inequality, we obtain 
\begin{multline*}
d^1_{\mathsf ZX}(A',B')<d^1_{\mathsf ZX}(A',A)+d^1_{\mathsf ZX}(A,B)+d^1_{\mathsf ZX}(B,B')\le 2\e+d^1_{\mathsf ZX}(A,B)<\\
2\e+d^1_{\mathsf ZY}(A,B)-7\e\le 
d^1_{\mathsf ZY}(A,A')+d^1_{\mathsf ZY}(A',B')+d^1_{\mathsf ZY}(B',B)-5\e<\\
\e+d^1_{\mathsf ZY}(A',B')+\e-5\e=d^1_{\mathsf ZY}(A',B')-3\e.
\end{multline*}
By Lemma~\ref{l:BBKZ}, there exists a finite graph $\Gamma=(V,E)\in\mathbf \Gamma^{\mathsf f}_{\!X\!}(A',B')$ such that $$\ell(\Gamma)<d^1_{\mathsf ZX}(A',B')+\e.$$ Since $Y$ is dense in $X$, we can find a function $f:V\to Y$ such that $f(x)=x$ if $x\in Y$ and $d_X(f(x),x)<\e/|E|$ if $x\in V\setminus Y$. Consider the graph $\Gamma'=(V',E')$ with the set of vertices $V'=f(V)$ and the set of edges $E'=\{\{f(x),f(y)\}:\{x,y\}\in E\}$. 
Observe that the graph $\Gamma'$ belongs to the family $\mathbf \Gamma_{\!Y\!}^{\mathsf f}(A',B')$ and hence
\begin{multline*}
d^1_{\mathsf ZY}(A',B')\le\ell(\Gamma')=\sum_{\{x',y'\}\in E'}d_X(x',y')\le\sum_{\{x,y\}\in E}d_X(f(x),f(y))\le \\
\sum_{\{x,y\}\in E}(d_X(f(x),x)+d_X(x,y)+d_X(y,f(y))<  
\sum_{\{x,y\}\in E}(\tfrac\e{|E|}+d_X(x,y)+\tfrac\e{|E|})<\\
2\e+\sum_{\{x,y\}\in E}d_X(x,y)=2\e+\ell(\Gamma)<2\e+d^1_{\mathsf ZX}(A',B')+\e<d^1_{\mathsf ZY}(A',B'),
\end{multline*} 
which is a desired contradiction showing that $d^1_{\mathsf ZX}(A,B)=d^1_{\mathsf ZY}(A,B)$.
\end{proof}

\begin{lemma} If $\bar X$ is a completion of $X$, then the complete metric space  $\mathsf Z^1\!\bar X$ is a completion of the metric space $\mathsf F^1\!X$.
\end{lemma}

\begin{proof} By Lemma~\ref{l:complete}, the metric space $\mathsf Z^1\!\bar X$ is complete. By Lemmas~\ref{l:BBKZ} and \ref{l:equal}, for any $A,B\in \mathsf FX$ we have 
$$d^1_{\mathsf FX}(A,B)=d^1_{\mathsf ZX}(A,B)=d^1_{\mathsf Z\bar X}(A,B),$$
so the metric space $\mathsf F^1\!X$ is a subspace of the complete metric space $\mathsf Z^1\!\bar X$. By Lemma~\ref{l:dense}, the space $\mathsf FX$ is dense in $\mathsf Z^1\!\bar X$.  This means that $\mathsf Z^1\!\bar X$ is a completion on $\mathsf F^1\!X$.
\end{proof}
 
Now we discuss the interplay between zero length and 1-dimensional Hausdorff measure. A subset $A$ of a metric space $X$ is defined to have {\em $1$-dimensional Hausdorff measure zero} if for any $\e>0$ there exists a countable set $C\subseteq X$ and a function $\epsilon:C\to(0,1]$ such that $\sum_{c\in C}\epsilon(c)<\e$ and $A\subseteq \bigcup_{c\in C}B(c,\epsilon(c))$. Here and further on by $$B(x,\delta)=\{y\in X:d_X(x,y)<\delta\}\mbox{ \ and \ }B[x,\delta]=\{y\in X:d_X(x,y)\le\delta\}$$ we denote respectively the open and closed balls of radius $\delta$ around a point $x$ in the metric space $(X,d_X)$.

\begin{proposition}\label{p:zero} If a subset $A$ of a metric space $(X,d_X)$ has zero length, then it is totally bounded, its closure has zero length and also $\bar A$ has 1-dimensional Hausdorff measure zero.
\end{proposition}

\begin{proof} If $A$ has zero length, then for every $\e>0$ there exists a graph $\Gamma=(V,E)$ in $X$ that has finitely many connected components such that $\ell(\Gamma)<\e$ and $A\subseteq\overline V$. Then also $\bar A\subseteq\overline V$, which means that $\bar A$ has zero length. To see that $\bar A$ has 1-dimensional Hausdorff measure zero, choose a finite set $D\subseteq V$ that meets each connected component of $V$ in a single point. Then $\{\Gamma(x)\}_{x\in D}$ is a finite disjoint cover of $V$. For every $x\in D$ let $\epsilon(x):=\sup_{y\in \Gamma(x)}d_X(x,y)$ and observe that $V\subseteq\bigcup_{x\in D}B(x,\epsilon(x))$. The connectedness of $\Gamma(x)$ implies that $\epsilon(x)\le\ell(\Gamma(x))$ and $\sum_{x\in D}\epsilon(x)\le\ell(\Gamma)<\e$. Choose any $\delta>0$ such that $|D|\cdot\delta+\sum_{x\in D}\epsilon(x)<\e$ and observe that
$$\bar A\subseteq\overline V\subseteq \bigcup_{x\in D}B[x,\epsilon(x)]\subseteq  \bigcup_{x\in D}B(x,\epsilon(x)+\delta).$$
Since $\sum_{x\in D}(\epsilon(x)+\delta)=|D|\cdot\delta+\sum_{x\in D}\epsilon(x)<\e$, and $\e$ is arbitrary, the set $\bar A$ has  1-dimensional Hausdorff measure zero.
\end{proof}

For subsets of the real line we have the following characterization.

\begin{proposition}\label{p:zero2} For a subset $A$ of the real line the following conditions are equivalent:
\begin{enumerate}
\item $A$ has zero length;
\item the closure $\bar A$ is compact and has zero length;
\item the closure $\bar A$ is compact and has $1$-dimensional Hausdorff measure zero;
\item the closure $\bar A$ is compact and has Lebesgue measure zero.
\end{enumerate}
\end{proposition}

\begin{proof} The implications $(1)\Ra(2)\Ra(3)$ were proved in Proposition~\ref{p:zero}. The implication $(3)\Ra(4)$ follows from the definition of the Lebesgue measure (as the 1-dimensional Hausdorff measure) on the real line.

To prove that $(4)\Ra(1)$, assume that the closure $\bar A$ is compact and has Lebesgue measure zero. Take any $\e>0$. Using the compactness of the set $\bar A$ and the regularity of the Lebesgue measure, 
construct inductively a decreasing sequence $(U_k)_{k\in\w}$ of bounded open  neighborhoods of $\bar A$ such that for every $k\in\w$ the following conditions are satisfied:
\begin{itemize}
\item $\overline U_{\!k+1}\subset U_k$;
\item the set $U_k$ has Lebesgue measure $\lambda(U_k)<\e/2^{k}$;
\item $U_k=\bigcup_{i=1}^{n_k}(a_{i,k},b_{i,k})$ for some $n_k\in\IN$ and real numbers $a_{1,k}<b_{1,k}\le\!\cdots\!\le a_{n_k,k}<b_{n_k,k}$ such that $A\cap (a_{i,k},b_{i,k})\ne\emptyset$ for every $i\in\{1,\dots,n_k\}$.
\end{itemize}
For every $k\in\w$ let $$a'_{i,k}:=\min\{a_{j,k+1}:j\in\{1,\dots,n_{k+1}\},\;a_{i,k}<a_{j,k+1}\}$$ and observe that $a'_{i,k}\le\min \big(\bar A\cap(a_{i,k},b_{i,k})\big)$ and hence $|a_{i,k}-a'_{i,k}|\le|a_{i,k}-b_{i,k}|$. 
For every $k\in\IN$, let $$\Omega_k=\big\{i\in\{1,\dots,n_k-1\}:\exists j\in\{1,\dots,n_{k-1}\}\;\;\;(b_{i,k},a_{i+1,k})\subseteq (a_{j,k-1},b_{j,k-1})\big\}.$$

Consider the graph $\Gamma=(V,E)$ with the set of vertices $$V=\bigcup_{k\in\w}\{a_{i,k},b_{i,k}:1\le i\le n_k\}$$ and the set of edges $$E=\big\{\{a_{i,k},b_{i,k}\},\{a_{i,k},a'_{i,k}\}:k\in\w,\;i\in\{1,\dots,n_k\}\big\}\cup\big\{\{b_{i,k},a_{i+1,k}\}:k\in\IN,\;i\in\Omega_k\big\}.$$
It is easy to see that $A\subseteq\bar A\subseteq \overline V$ and each connected component of the graph $\Gamma$ intersects the set $\{a_{i,0}:1\le i\le n_0\}$. Therefore, $\Gamma$ has finitely many connected components. Also
\begin{multline*}
\ell(\Gamma)\le \sum_{k=0}^\infty \sum_{i=1}^{n_k}(|b_{i,k}-a_{i,k}|+|a'_{i,k}-a_{i,k}|)+\sum_{k=1}^\infty \sum_{i\in\Omega_k}|a_{i+1,k}-b_{i,k}|<\\
 2\cdot\sum_{k=0}^\infty \sum_{i=1}^{n_k}|b_{i,k}-a_{i,k}|+\sum_{k=1}^\infty\sum_{j=1}^{n_{k-1}}|b_{i,k-1}-a_{i,k-1}|= 3\cdot\sum_{k=0}^\infty \sum_{i=1}^{n_k}|b_{i,k}-a_{i,k}|\le\\
3\cdot\sum_{k=0}^\infty\lambda(U_k)<3\sum_{k=0}^\infty\frac\e{2^{k}}=3\e,
 \end{multline*}
 which implies that the set $A$ has zero length.
\end{proof}

\begin{proposition}\label{p:R} For the real line $X=\IR$, the identity inclusion $\mathsf Z^1\!X\to \mathsf KX$ is a topological embedding. 
\end{proposition}

\begin{proof} Because of Lemma~\ref{l:H}, it suffices to prove that for every $A\in\mathsf ZX$ and $\e>0$ there exists $\delta>0$ such that for any $B\in\mathsf ZX$ the inequality $d_{\mathsf KX}(A,B)<\delta$ implies $d^1_{\mathsf ZX}(A,B)<\e$.

By Proposition~\ref{p:zero2},  the set $\bar A$ is compact and has Lebesgue measure zero. By the regularity of the Lebesgue measure on the real line, there exists an open neighborhood $U$ of  $\bar A$ in $\IR$ such that $U=\bigcup_{i=1}^n(a_i,b_i)$ for some sequence $a_1<b_1<a_2<b_2<\dots<a_n<b_n$ such that $\sum_{i=1}^n|b_i-a_i|<\tfrac19\e$. By the proof of Proposition~\ref{p:zero2}, there exists a graph $\Gamma_{\!A}=(V_{\!A},E_A)$ such that $\bar A\subseteq \overline V_{\!\!A}$, $\ell(\Gamma_{\!A})<3\cdot\tfrac19\e=\tfrac13\e$, and each connected component of $\Gamma_{\!A}$ intersects the set $\{a_i\}_{i=1}^n$. Find $\delta>0$ such that every set $B\in\mathsf KX$ with $d_{\mathsf KX}(A,B)<\delta$ is contained in $U$. Take any set $B\in\mathsf ZX$ with $d_{\mathsf KX}(A,B)<\delta$. Then $B\subseteq U$ and by the proof of Proposition~\ref{p:zero2}, there exists a graph $\Gamma_{\!B}=(V_{\!B},E_{\!B})$  with finitely many components such that $\overline B\subseteq\overline V_{\!B}\subset U$ and $\ell(\Gamma_{\!B})<3\cdot\frac19 \e=\frac13\e$. Let $D\subseteq V$ be a finite set intersecting each connected component of the graph $\Gamma_{\!B}$. 

For every $i\in\{1,\dots,n\}$, write the set $\{a_i\}\cup \big(D\cap (a_i,b_i)\big)$ as $\{a_{i,0},\dots,a_{i,m_i}\}$ for some points $a_{i,0}<\dots<a_{i,m_i}$. It follows that $a_{i,1}=a_i$ and $a_{i,m_i}\le b_i$, which implies $\sum_{j=1}^{m_i}|a_{i,j}-a_{i,j-1}|\le|b_i-a_i|$.
Consider the graph $\Gamma=(V,E)$ with the set of vertices $V=V_A\cup V_B$ and the set of edges  
$$E=E_A\cup E_B\cup\bigcup_{i=1}^n\big\{\{a_{i,j-1},a_{i,j}\}:j\in\{1,\dots,m_i\}\big\}.$$
It can be shown that $\Gamma\in\mathbf\Gamma_{\!X\!}(A,B)$ and hence
$$d^1_{\mathsf ZX}(A,B)\le\ell(\Gamma)\le\ell(\Gamma_A)+\ell(\Gamma_B)+\sum_{i=1}^n\sum_{j=1}^{m_i}|a_{i,j}-a_{i,j-1}|<\tfrac13\e+\tfrac13\e+\sum_{i=1}^n|b_i-a_i|<\tfrac23\e+\tfrac19\e<\e.
$$  
\end{proof}

Proposition~\ref{p:R} is specific for the real line and does not hold for higher-dimensional Euclidean spaces. To prove this fact, let us recall the definition of the upper box-counting dimension $\overline{\dim}_B(X)$ of a metric space $X$. Given any $\e>0$, denote by $N_\e(X)$ by the smallest cardinality of a cover of $X$ by subsets of diameter $\le \e$. Observe that the metric space $X$ is totally bounded iff $N_\e(X)$ is finite for every $\e>0$. If $X$ is not totally bounded, then put $\overline{\dim}_B(X)=\infty$. If $X$ is totally bounded, then let
$$\overline{\dim}_B(X):=\limsup_{\e\to+0}\frac{\ln N_\e(X)}{\ln(1/\e)}\in[0,\infty].$$
By \cite[\S3.2]{Fal}, for every $n\in\IN$, every bounded set $X\subseteq\IR^n$ with nonempty interior has $\overline{\dim}_B(X)=n$.

In the following proposition we endow the hyperspace $\mathsf FX$ with the Hausdorff metric.

\begin{proposition}\label{p:dim} Let $X$ be a metric space and $Y\subseteq X$ be a subspace of $X$ such that $\overline{\dim}_B(Y)>1$. Then for any $l\in\IN$ there exists a nonempty  finite subset $A\subseteq Y$ such that $d^1_{\mathsf FX}(A,\{x\})\ge l$ for any singleton $\{x\}\subseteq X$.
\end{proposition}

\begin{proof} To derive a contradiction, assume that there exists $l\in\IN$ such that for any finite set $A\subseteq Y$ there exists $x\in X$ such that $d^1_{\mathsf FX}(A,\{x\})< l$.  

We are going to show that $N_{2\e}(Y)\le (2l+1)/\e$ for every $\e\in(0,1]$.
Given any $\e\in(0,1]$, use the Kuratowski-Zorn Lemma and find a  maximal subset $M$ in $Y$, which is $2\e$-separated in the sense that  $d_X(y,z)\ge 2\e$ for any distinct points $y,z\in M$. The maximality of the set $M$ implies that $Y\subseteq  \bigcup_{y\in M}B(y,2\e)$.

We claim that $|M|\le (1+2l)/\e$. To derive a contradiction, assume that $|M|>(1+2l)/\e$. In this case we can find a finite subset $A\subseteq M$ such that $|A|>(1+2l)/\e$. The choice of the number $l$ ensures that $d^1_{ZX}(A,\{x\})< l$ for some $x\in X$. By Lemma~\ref{l:BBKZ}, there exists a finite graph $\Gamma\in\mathbf\Gamma_{\!X\!}(\{x\},A)$ such that $\ell(\Gamma)<l$. Since each connected component of the graph $\Gamma$ meets the singleton $\{x\}$, the graph $\Gamma=(V,E)$ is connected. Replacing $\Gamma$ by a minimal connected subgraph, we can assume that $\Gamma$ is a tree.

By Lemma~\ref{l:tree} (proved below), there exists a sequence $v_0,\dots,v_n\in V$ such that 
\begin{itemize}
\item[(i)] $V=\{v_0,\dots,v_n\}$;
\item[(ii)] $\big\{\{v_{i-1},v_i\}:1\le i\le n\big\}\subseteq E$;
\item[(iii)] for every $e\in E$ the set $\big\{i\in\{1,\dots,n\}:\{v_{i-1},v_i\}=e\big\}$ contains at most two elements.
\end{itemize}

Choose a sequence of real numbers $t_0,\dots,t_n$ such that $t_0=0$ and $t_i-t_{i-1}=d_X(v_i,v_{i-1})$ for every $i\in\{1,\dots,n\}$.   The condition (iii) implies that $t_n\le 2\ell(\Gamma)<2l$. Then the set $T=\{t_0,\dots,t_n\}$ has $$N_\e(T)<1+\frac{t_n}\e<1+\frac{2l}{\e}\le \frac{1+2l}{\e}.$$ Taking into account that the map $T\to V$, $t_i\mapsto v_i$, is non-expanding, we conclude that $N_\e(A)\le N_\e(V)\le N_\e(T)< (1+2l)/\e$. Since the set $A$ is $2\e$-separated, it has cardinality $|A|=N_\e(A)<(1+2l)/\e$, which contradicts the choice of $A$. 

 This contradiction shows that $|M|\le (1+2l)/\e$ and then $N_{2\e}(Y)\le |M|\le (1+2l)/\e$ for any $\e>0$. Taking the upper limit at $\e\to+0$, we obtain the upper bound
$$\overline{\dim}_B(Y)=\limsup_{\e\to+0}\frac{\ln N_\e(Y)}{\ln(1/\e)}=\limsup_{\e\to+0}\frac{\ln N_{2\e}(Y)}{-\ln(1/(2\e))}\le\limsup_{\e\to+0}\frac{\ln((1+2l)/\e)}{\ln(1/(2\e))}=1,$$
which contradicts our assumption. 
\end{proof}

\begin{lemma}\label{l:tree} For any finite tree $\Gamma=(V,E)$, there exists a sequence $v_0,\dots,v_n\in V$ such that
\begin{itemize}
\item[(i)]  $V=\{v_0,\dots,v_n\}$, 
\item[(ii)] $\big\{\{v_{i-1},v_i\}:1\le i\le n\big\}=E$, and
\item[(iii)]  for every edge $e\in E$ the set $\{i\in\{1,\dots,n\}:\{v_{i-1},v_i\}=e\}$ contains at most two elements.
\end{itemize}
\end{lemma}

\begin{proof} This lemma will be proved by induction on the cardinality $|V|$ of the tree $V$. If $|V|=1$, then let $v_0$ be the unique vertex of $X$ and observe that the sequence $v_0$ has the properties (i)--(iii). Assume that for some $k\ge 2$ the lemma has been proved for all trees on $<k$ vertices. Let $\Gamma=(V,E)$ be any tree with $|V|=k$. By \cite[1.5.1]{Diestel}, the tree $\Gamma$ has  exactly $k-1$ edges. Consequently, there exists a vertex $v\in V$ having a unique neighbor $u\in V\setminus\{v\}$ in the tree $(V,E)$. Put $V'=V\setminus\{v\}$, $E'=E\setminus\big\{\{u,v\}\big\}$ and observe that $(V',E')$ is a tree on $k-1$ vertices. By the inductive assumption, there exists a sequence $v'_1,\dots,v'_n\in V'$ such that $V'=\{v'_1, \dots, v'_n\}$, $\big\{\{v'_{i-1},v'_i\}:i\in \{1,\dots,n\}\big\}=E'$, and for every $e\in E'$ the set $\{i\in\{1,\dots,n\}:\{v'_{i-1},v'_i\}=e\}$ contains at most two elements.

Find an index $j\in\{1,\dots,n\}$ such that $v'_j=u$ and consider the sequence $v_0,\dots,v_{n+1}$, where $v_i=v'_i$ for $i\le j$, $v_{j+1}=v$, and $v_{i}=v'_{i-2}$ for $i\in \{j+1,\dots,n+2\}$.  It is easy to see that the sequence $v_0,\dots,v_{n+2}$ has the properties (i)--(iii).
\end{proof}

Proposition~\ref{p:dim} implies the following corollary, in which by $\mathsf FX$ we denote the hyperspace of nonempty finite subsets of $X$, endowed with the Hausdorff metric.

\begin{corollary} Let $X$ be a metric space. If for some point $x\in X$ the identity map $\mathsf FX\to\mathsf F^1\!X$ is continuous at $\{x\}$, then the point $x$ has a neighborhood $O_x\subseteq X$ with box-counting dimension $\overline{\dim}_B(O_x)\le 1$.
\end{corollary} 

\begin{proof} Assuming that the identity map $\mathsf FX\to\mathsf Z^1\!X$ is continuous at $\{x\}$, we can find $\delta>0$ such that for any set $A\in\mathsf FX$ with $d_{\mathsf FX}(A,\{x\})<\delta$ we have $d_{\mathsf FX}^1(A,\{x\})<1$. Let $O_x:=B(x,\delta)$. Assuming that $\overline{\dim}_B(O_x)>1$, we can apply Proposition~\ref{p:dim} and find a finite set $A\subseteq O_x$ such that $d^1_{\mathsf FX}(A,\{x\})>1$. On the other hand, the inclusion $A\subseteq O_x=B(x,\delta)$ implies that $d_{\mathsf FX}(A,x)<\delta$ and hence $d_{\mathsf FX}^1(A,\{x\})<1$ by the choice of $\delta$. This contradiction shows that $\overline{\dim}_B(O_x)\le 1$.
\end{proof}

Finally, we present an example showing that the equivalence $(2)\Leftrightarrow(3)$ in Proposition~\ref{p:zero2} does not hold for higher-dimensional Euclidean spaces.

\begin{example} Assume that $X$ is a complete metric space such that every nonempty open set $U\subseteq X$ has box-counting dimension $\overline{\dim}_B(U)>1$. Then every nonempty open set $U$ contains a compact subset $A\subseteq U$ such that $A$ has $1$-dimensional Hausdorff measure zero but fails to have zero length.
\end{example}

\begin{proof} 
Choose any point $x_0\in U$ and a positive number $\e_0$ such that $B[x_0,\e_0]\subseteq U$. Put $A_0=\{x_0\}$. For every $n\in\IN$ we shall inductively choose a finite subset $A_n\subseteq X$, a positive real number $\e_n$, and a map $r_n:A_n\to A_{n-1}$, satisfying the following conditions:
\begin{itemize}
\item[(i)] $A_{n-1}\subseteq A_n$;
\item[(ii)] $\e_n\le \frac1{2^n|A_n|}$;
\item[(iii)] $B[x,\e_n]\cap B[y,\e_n]=\emptyset$ for any distinct points $x,y\in A_n$;\item[(iv)] $r_n(x)=x$ for any $x\in A_{n-1}$;
\item[(v)] $B[x,\e_n]\subseteq B(r_n(x),\e_{n-1})$ for any $x\in A_{n-1}$;
\item[(vi)] $d^1_{\mathsf FX}(\{x\},r_n^{-1}(x))>n$ for every $x\in A_{n-1}$.
\end{itemize}
Assume that for some $n\in\IN$ we have constructed a set $A_{n-1}$ and a number $\e_{n-1}>0$ satisfying the condition (iii). By our assumption, for every $y\in A_{n-1}$ the ball $B(y,\e_{n-1})$ has $\overline{\dim}_BB(y,\e_{n-1})>1$.  By Proposition~\ref{p:dim}, the ball $B(y,\e_{n-1})$ contains a finite subset $A'_y$ such that $d^1_{\mathsf FX}(A'_y,\{y\})>n$. The definition of the metric $d^1_{\mathsf FX}$ implies that $d^1_{\mathsf FX}(A'_y\cup\{y\},\{y\})=d^1_{\mathsf FX}(A'_y,\{y\})>n$.  Let $A_n=\bigcup_{y\in A_{n-1}}(\{y\}\cup A'_y)$ and $r_n:A_n\to A_{n-1}$ be the map assigning to each point  $x\in A_n$ the unique point $y\in A_{n-1}$ such that $x\in A'_y\cup\{y\}$. It is clear that the $A_n$ satisfies the inductive condition (i) and the function $r_n$ satisfies the conditions (iv), (vi). Now choose any number $\e_n$ satisfying the conditions (ii), (iii) and (v). This completes the inductive step.

After completing the inductive construction, consider the compact set
$$A=\bigcap_{n\in\w}\bigcup_{x\in A_n}B[x,\e_n]\subseteq U$$ in $X$.
We claim that the set $A$ has 1-dimensional Hausdorff measure zero. Given any $\e>0$, find $n\in\w$ such that $\frac2{2^n}<\e$ and observe that $A\subseteq\bigcup_{x\in A_n}B(x,2\e_n)$ and $$\sum_{x\in A_n}2\e_n<\sum_{x\in A_n}\frac2{2^n|A_n|}=\frac2{2^n}<\e,$$
witnessing that the 1-dimensional Hausdorff measure of $A$ is zero.

Assuming that $A$ has zero length, we  calculate the distance $d^1_{\mathsf ZX}(A,A_0)<\infty$ and find a graph $\Gamma\in\mathbf \Gamma_{\!X\!}(A,A_0)$ such that $\ell(\Gamma)<\infty$. Since each component of $\Gamma$ intersects the singleton $A_0=\{x_0\}$, the graph $\Gamma$ is connected. Take any integer number $n>\ell(\Gamma)$ and conclude that for every $x\in A_{n-1}$ we have $\{x\}\cup r_n^{-1}(x)\subseteq A\subseteq\overline V$ and hence $\Gamma\in\mathbf\Gamma_{\!X\!}(\{x\},r_n^{-1}(x))$. By Lemma~\ref{l:BBKZ},
$$d^1_{\mathsf FX}(\{x\},r_n(x))=d^1_{\mathsf ZX}(\{x\},r_n(x))\le\ell(\Gamma)<n,$$
which contradicts the inductive condition (vi). This contradiction shows that the set $A$ fails to have zero length.
\end{proof}

\begin{remark} There are interesting algorithmic problems related to efficient calculating the distance $d^1_{\mathsf FX}(A,B)$ between nonempty finite subsets $A,B$ of a metric space. For a nonempty finite subset $A$ of the Euclidean plane $\IR^2$ and a singleton $B=\{x\}\subset\IR^2$, the problem of calculating the distance $d^1_{\mathsf FX}(A,B)$ reduces to the classical Steiner's problem \cite{Steiner} of finding a tree of the smallest length that contains the set $A\cup B$. This problem is known \cite{Holby} to be computationally very difficult. On the other hand, for  nonempty finite subsets  of the real line, there exists an efficient algorithm \cite{MO} of complexity $O(n\ln n)$ calculating the distance $d^1_{\mathsf F\IR}(A,B)$ between two sets $A,B\in\mathsf F\mathbb R$ of cardinality $|A|+|B|\le n$. Also there exists an algorithm of the same complexity $O(n\ln n)$ calculating the Hausdorff distance $d_{\mathsf F\IR}(A,B)$ between the sets $A,B$. Finally, let us remark that the evident brute force algorithm for calculating the Hausdorff distance $d_{\mathsf FX}(A,B)$ between nonempty finite subsets of an arbitrary metric space $(X,d_X)$  has complexity $O(|A|{\cdot}|B|)$. Here we assume that calculating the distance between points requires a constant amount of time.
\end{remark}


\begin{thebibliography}{}

\bibitem{MO} T.~Banakh, {\em A quick algorithm for calculating the $\ell^1$-distance between two finite sets on the real line}, (https://mathoverflow.net/a/277928/61536), 04.08.2017.

\bibitem{BBKZ} T.~Banakh,  V.~Brydun, L.~Karchevska, M.~Zarichnyi, {\em The $\ell^p$-metrization of functors with finite supports}, preprint.

\bibitem{Beer} G.~Beer, {\em Topologies on closed and closed convex sets}, Kluwer Academic Publishers Group, Dordrecht, 1993.

\bibitem{Steiner} M.~Brazil, R.~Graham, D.A.~Thomas, M.~Zachariasen, {\em On the history of the Euclidean Steiner tree problem}, Arch. Hist. Exact Sci. {\bf 68}:3 (2014),  327--354.

\bibitem{Diestel} R.~Diestel, {\em Graph Theory}, GTM {\bf 173}. Springer-Verlag, Berlin, 2005.

\bibitem{Edgar} G.~Edgar, {\em Measure, Topology, and Fractal Geometry}, Springer, New York, 2008.

\bibitem{Eng} R.~Engelking, {\em General Topology}, Heldermann Verlag, Berlin, 1989.

\bibitem{Fal} K.~Falconer, {\em Fractal Geometry. Mathematical foundations and applications}, John Wiley \& Sons, Inc., Hoboken, NJ, 2003.

\bibitem{Holby} M.W.~Bern, R.~Graham,{\em The shortest-network problem},  Scientific American. {\bf 260}:1 (1989) 84--89.


\end{thebibliography}
\end{document}